\documentclass[12pt]{article}
\usepackage{amssymb,amsmath,amsfonts}
\usepackage{amsthm}
\usepackage{color}
\usepackage{mathabx}

\numberwithin{equation}{section} \theoremstyle{plain}
\newtheorem{theorem}{Theorem}[section]
\newtheorem{corollary}[theorem]{Corollary}
\newtheorem{proposition}[theorem]{Proposition}
\newtheorem{lemma}[theorem]{Lemma}

\theoremstyle{definition}
\newtheorem{definition}{Definition}[section]
\newtheorem{example}{Example}[section]
\theoremstyle{remark}
\newtheorem{remark}{\rm\bf Remark}[section]

\allowdisplaybreaks

\begin{document}

\title{On the $q$-moment determinacy of probability distributions}
\author{Sofiya Ostrovska and Mehmet Turan}
\date{}
\maketitle

\begin{center}
{\it Atilim University, Department of Mathematics,  Incek  06836, Ankara, Turkey}\\
{\it e-mail: sofia.ostrovska@atilim.edu.tr, mehmet.turan@atilim.edu.tr}\\
{\it Tel: +90 312 586 8211,  Fax: +90 312 586 8091}
\end{center}

\begin{abstract}

 Given $0<q<1,$ every absolutely continuous distribution can be described in two different ways: in terms of a probability density function and also  in terms of a $q$-density. Correspondingly, it has a sequence of moments and a sequence of $q$-moments if those exist. In this article, new conditions on the $q$-moment determinacy of probability distributions are derived. In addition, results related to the comparison of the properties of probability distributions with respect to the moment and $q$-moment determinacy are presented.
\end{abstract}



\section{Introduction}

Due to the popularity of the $q$-calculus, numerous $q$-analogues of classical probability distributions have emerged,
both for discrete and absolutely continuous cases. For example, there are $q$-binomial, $q$-Poisson, $q$-exponential, $q$-Erlang and other $q$-distributions. These distributions play a significant role not only in the $q$-calculus itself, but also in various applications, primarily in theoretical physics. See, for example, \cite{biederhan, charal, jing}. Comprehensive information concerning $q$-distributions is presented in \cite{charal}, and in this article we follow the terminology and exposition of the monograph. Throughout the paper, all random variables are taken to be non-negative and $0<q<1.$ Also, the $q$-integral defined by Jackson for $0<a<b$ as 
\begin{align*}
\int_0^a f(t) d_qt = a(1-q)\sum_{j=0}^{\infty} f(aq^j)q^j, \quad 
\int_a^b f(t) d_qt=\int_0^b f(t) d_qt - \int_0^a f(t) d_qt
\end{align*} will be used along with the improper $q$-integral on $[0,+\infty)$ defined as
\begin{align}\label{qimpr}
\int_0^\infty f(t) d_qt = (1-q)\sum_{j=-\infty}^{\infty} f(q^j)q^j.
\end{align}
See \cite[Sec. 19]{quantum}. 
\begin{definition} \cite{charal} Let $X$ be a random variable with distribution $P_X$ and distribution function $F_X.$ 
A function $f(t),$ $t>0,$ is a {\it $q$-density} of $X$ if 
\begin{align}
F_X(x)=\int_0^x f(t) d_qt, \quad x>0. \label{qdens}
\end{align}
Correspondingly, the $n$-th order $q$-moment of $X$ is
\begin{align}
m_q(n; X):=m_q(n;f):=\int_0^\infty t^nf(t) d_qt, \quad k\in{\mathbb N}_0. \label{qmom}
\end{align}
\end{definition}

 It has to be mentioned here that if $P_X$ has a $q$-density $f$, then $f$ is the $q$-derivative of the distribution function $F_X,$ that is,
\begin{align}\label{qder}f(t)=D_qF_X(t):=\frac{F_X(t)-F_X(qt)}{t(1-q)}, \quad t>0.\end{align} It is known (\cite[Theorem 20.1]{quantum}) that if $X\geqslant 0$ and $F_X$ is continuous at 0, then $F_X$ can be represented in the form \eqref{qdens} and, therefore, possesses a $q$-density.

The moment problem for the $q$-moments in terms of $q$-densities has been considered in 
\cite{stapro}. 
Since the $q$-moments depend only on the values a $q$-density on the sequence $\{q^j\}_{j\in{\mathbb Z}},$ it is reasonable, therefore, to consider the following equivalence relation for functions on $(0, \infty):$
$$
f\sim g \Leftrightarrow f(q^j)=g(q^j), \quad j\in{\mathbb Z}.
$$
 Notice that $q$-moments $m_q(n;f)$ may also be obtained as the moments of a discrete distribution concentrated on $\{q^j\}_{j\in{\mathbb Z}},$ whose probability mass function is given as:$${\mathbf{P}}\{X=q^j\}=f(q^j)q^j(1-q),\quad j\in \mathbb{Z}.$$ The moment problem for such discrete distributions was investigated in \cite{berg} by C. Berg, who found explicitly infinite families of distributions all possessing the same moments of all orders.  These families can also be viewed as discrete Stieltjes classes, although the name `Stieltjes class' was suggested by J. Stoyanov (\cite{jap}) a few years after \cite{berg} had been published.  

\begin{definition}\cite{stapro}
A distribution $P_X$ of a random variable $X$ possessing a $q$-density $f$ is moment determinate if $m_q(k;X)=m_q(k;Y)$ for all $k\in \mathbb{N}_0$   implies that $f_X \sim f_Y.$ Otherwise, $P_X$ is $q$-moment indeterminate.
\end{definition}
It should be pointed out that every absolutely continuous distribution possessing finite moments of all orders can be examined from two different perspectives: those of moment determinacy and $q$-moment determinacy. 

In \cite{stapro}, some conditions have been provided both for $q$-moment determinacy and indeterminacy in terms of the values $f(q^{-j}).$ More precisely, it has been proved that 
\begin{enumerate}
\item [(i)] if 
\begin{align}\label{B}
f(q^{-j})=o(q^{j(j+1)/2}), \quad j\to\infty,
\end{align}
then $P_X$ is $q$-moment determinate;
\item [(ii)] if 
\begin{align}\label{C}
f(q^{-j}) \geqslant Cq^{j(j+1)/2}, \quad j\geqslant 0,
\end{align}
then $P_X$ is $q$-moment indeterminate.
\end{enumerate} 
Statement (i) implies immediately that if a $q$-density $f$ has a bounded support, then the distribution $P_X$ is $q$-moment determinate. 

In this work, new results on $q$-moment (in)determinacy are presented, both in terms of $q$-moments and $q$-density itself. Alternatively, it can be stated that some `checkable' conditions for $q$-moment (in)determinacy are given. For the classical moment problem, an extensive review of such conditions can be found in \cite{recent}. To illustrate the difference between the notions of moment and $q$-moment determinacy, examples of probability distributions which are moment indeterminate but 
 at the same time
$q$-moment determinate are provided. The exact relation between the two notions is yet to be described. As a first attempt, the  outcomes connecting these two aspects are presented in Propositions \ref{propcon} and \ref{propqexp}. 

The $q$-analogue of the exponential function
\begin{align*}
e_q(t)=\prod_{j=0}^\infty \left(1-t(1-q)q^j\right)^{-1}
\end{align*}
is used in the paper. For ample information on $e_q(t),$ we refer to \cite[Section 1]{charal} and \cite[Section 9]{quantum}. The $q$-exponential function is involved in the $q$-density of the $r$-stage Erlang distribution of the first kind with parameter $\lambda >0:$
\begin{align}
f_r(t)=\frac{q^{r(r-1)/2}\lambda^r}{[r-1]_q !} \, t^{r-1} e_q(-\lambda t), \quad t>0. \label{erl}
\end{align}
See \cite[formula (2.24)]{charal}. A stochastic process leading to this distribution as well as some of its properties have been studied in \cite{kyriakoussis}. When $r=1,$ one recovers the $q$-exponential distribution with parameter $\lambda >0,$ whose density is:
\begin{align}
f(t)=\lambda e_q(-\lambda t), \quad t>0. \label{qexp}
\end{align}
See \cite[Corollary 2.1, p. 77]{charal}. It will be shown that the $q$-moment determinacy of a $q$-Erlang distribution depends on the values of $\lambda$ and $r.$ Observe that for the classical $r$-stage Erlang distribution this is not the case as it is moment determinate for all $\lambda >0$ and $r\in{\mathbb N}.$ This uncovers the difference between the problems of moment and $q$-moment determinacy.

For the sequel, we need the well-known Euler's identity \cite[Section 9]{quantum}:
\begin{align}
\prod_{j=0}^{\infty}(1+q^j t)=\sum_{j=0}^{\infty} \frac{q^{j(j-1)/2}}{(q;q)_j}\, t^j, \quad t\in{\mathbb C} \label{eu}
\end{align}
where $(q;q)_j$ is the $q$-shifted factorial defined by:
\begin{align*}
(a;q)_0:=1, \quad (a;q)_j=\prod_{s=0}^{j-1}(1-aq^s), \quad a\in{\mathbb C}.
\end{align*}
The following estimate proved in \cite[formula (2.6)]{zeng} holds for some positive constants $C_1,$ $C_2$ and $t$ large enough:
\begin{align}
C_1 \exp\left\{\frac{\ln^2 t}{2\ln(1/q)}+\frac{\ln t}{2}\right\} \leqslant 
\prod_{j=0}^{\infty}(1+q^j t) \leqslant 
C_2 \exp\left\{\frac{\ln^2 t}{2\ln(1/q)}+\frac{\ln t}{2}\right\}. \label{zeng}
\end{align}
Throughout the paper, the letter $C$ with or without an index stands for a positive constant whose exact value does not have to be specified. Additionally, the notation $M(r;f):=\max_{|z|=r} |f(z)|$, where $f(z)$  is a function analytic in $\{z:|z|=r\}$ will be used repeatedly.

\section{Statement of results}
We start with the assertion providing an analogue of condition \eqref{B} proved in \cite[Theorem 2.4]{stapro} in terms of $q$-moments.
\begin{proposition}\label{prop1}
Let $P_X$ have a $q$-density $f$ and $m_q(n;f)$ be finite $q$-moments for all $n\in{\mathbb N}.$
If 
\begin{align}
\limsup_{n\to\infty} \frac{\ln m_q(n;f)}{n^2} = A < \frac{\ln(1/q)}{2}, \label{A}
\end{align}
then $P_X$ is $q$-moment determinate.
\end{proposition}
\begin{remark} If $A=\ln(1/q)/2,$ then the distribution $P_X$ may be either $q$-moment determinate of $q$-moment indeterminate. Hence, the bound $A$ in \eqref{A} cannot be improved. This will be illustrated in Example \ref{exABC}.
\end{remark}

To establish conditions for $q$-moment indeterminacy, we have to impose some restrictions on the behaviour of a $q$-density. The following statement holds.
\begin{theorem}\label{thm2}
Let $\{m_q(n;f)\}_{n=1}^{\infty}$ be a sequence of $q$-moments of a distribution $P_X.$ If 
\begin{align}
\frac{\ln(1/q)}{2} < \limsup_{n\to\infty} \frac{\ln m_q(n;f)}{n^2} =A < \infty, \label{beta}
\end{align}
and
\begin{align}
f(q^{-j-1}) f(q^{-j+1}) \leqslant [f(q^{-j})]^2, \quad j\geqslant 0, \label{lc}
\end{align}
then $P_X$ is $q$-moment indeterminate.
\end{theorem}

\begin{remark}
Condition \eqref{lc} shows that the sequence $\{f(q^{-j})/f(q^{-(j+1)})\}$ is non-decreasing, that is the sequence $\{f(q^{-j})\}$ is log-concave. The logarithmic concavity plays an important role in the study of probability distributions. 
\end{remark}
The next result provides a condition for the $q$-moment indeterminacy in the situations not covered   by the  outcomes of Proposition \ref{prop1} and Theorem \ref{thm2}.

\begin{theorem}\label{thm3} Let $f(t)$ be a $q$-density of a random variable $X$ and $m\in{\mathbb N}.$ If 
\begin{align}
f(q^{-mj}) \geqslant C q^{mj(j+1)/2}, \quad j \geqslant 0, \label{mj}
\end{align}
then the distribution $P_X$ is $q$-moment indeterminate.
\end{theorem}

\begin{proposition}\label{propcon}
Let $X\geqslant 0$ possess an absolutely continuous distribution.    If $\{\mu_n\}_{n=1}^{\infty}$ is a sequence of moments of $X,$ and
\begin{align}
\limsup_{n\to\infty} \frac{\ln \mu_n}{n^2} = \frac{\ln(1/q_0)}{2}, \label{q1}
\end{align}
then $P_X$ is $q$-moment determinate for all $q < q_0.$ In particular, if 
\begin{align}
\limsup_{n\to\infty} \frac{\ln \mu_n}{n^2}=0, \label{q2}
\end{align}
then $P_X$ is $q$-moment determinate for all $q\in(0,1).$
\end{proposition}
It should be emphasized that condition \eqref{q2} is not conclusive to the moment (in)determinacy, while it guarantees the $q$-moment determinacy for $q\in(0,1).$ This is illustrated by Example \ref{hyper}.
The next assertion deals with the $q$-exponential distribution, which has $q$-density \eqref{qexp}.

\begin{proposition}\label{propqexp}
Let $X$ be a random variable whose distribution function is $F_X(t)=1-e_q(-\lambda t),$ $t\geqslant 0.$ Then, the distribution $P_X$ is moment indeterminate for all $\lambda >0.$
\end{proposition}

Recall that it was proved in \cite[Example 2.1]{stapro} that the $q$-exponential distribution is $q$-moment determinate when $\lambda(1-q)>1$ and $q$-moment indeterminate otherwise. Juxtaposing this claim with  Propositon \ref{propqexp}, the following conclusion can be reached.

\begin{corollary} There exist absolutely continuous probability distributions which are  moment indeterminate but $q$-moment determinate.
\end{corollary}

\section{Proofs of the results}
The next lemma proved in \cite[Lemma 2.6]{stapro} will be used in the sequel.
\begin{lemma}\label{lem1} Let $\phi(z)=\sum_{j\in{\mathbb Z}} c_j z^j$ satisfy $\phi(q^{-m})=0$ for all $m\in{\mathbb N}.$ Then, for $r=q^{-m},$ one has
$$
M(r; \phi) \geqslant C \exp\left\{ \frac{\ln^2 r}{2\ln(1/q)}-\frac{\ln r}{2}\right\}.
$$
\end{lemma}
\begin{proof}[Proof of Propositon \ref{prop1}]
Assume that there exists a $q$-density $g\not\sim f$ such that $m_q(n;f)=m_q(n;g)$ for all $n\in{\mathbb N}_0,$ that is,
$$
\sum_{j\in{\mathbb Z}} f(q^{-j}) q^{-nj}=\sum_{j\in{\mathbb Z}} g(q^{-j}) q^{-nj}, \quad n\in{\mathbb N}.
$$
The existence of the $q$-moments implies that the Laurent series $\sum_{j\in{\mathbb Z}} f(q^{-j})z^j$ and $\sum_{j\in{\mathbb Z}} g(q^{-j})z^j$ converge in ${\mathbb C}^*={\mathbb C}\setminus \{0\}$ to $\phi_1(z)$ and $\phi_2(z),$ respectively, both of which are analytic in ${\mathbb C}^*.$
Then, by Lemma \ref{lem1}, for $\phi=\phi_1-\phi_2,$ one has:
\begin{align}
M(r;\phi) \geqslant C \exp\left\{ \frac{\ln^2 r}{2\ln(1/q)}-\frac{\ln r}{2}\right\}, \quad r={q^{-n}}. \label{mrphi}
\end{align}
On the other hand, for $r=q^{-n},$
\begin{align*}
M(r;\phi) \leqslant \frac{2 m_q(n-1;f)}{1-q}, \quad n\in{\mathbb N}. 
\end{align*}
Hence,
\begin{align}
\limsup_{n\to\infty}\frac{\ln M(r;\phi)}{n^2} \leqslant
\limsup_{n\to\infty}\frac{\ln m_q(n;f)}{n^2} < \frac{\ln(1/q)}{2}, \label{less} 
\end{align}
due to the assumption \eqref{A}. Meanwhile, \eqref{mrphi} yields
\begin{align*}
\limsup_{n\to\infty} \frac{\ln M(q^{-n};\phi)}{n^2} \geqslant \limsup_{n\to\infty}
\frac{n^2\ln(1/q)-n\ln(1/q)}{2n^2} = \frac{\ln(1/q)}{2}
\end{align*}
which contradicts \eqref{less}.
\end{proof}
\begin{example}\label{exABC}
Consider $r$-stage Erlang distribution of the first kind with parameter $\lambda$ whose density is given by \eqref{erl}. Applying the conditions \eqref{B} and \eqref{C}, one can derive that if $q^r (1-q)\lambda \leqslant 1,$ then the distribution is $q$-moment indeterminate, and if $q^r (1-q)\lambda > 1,$ then the distribution is $q$-moment determinate. Consequently, for every $\lambda >0,$ the distribution becomes $q$-moment indeterminate when the number of stages is large enough.
\end{example}
The proof of Theorem \ref{thm2} is based on the following result of V. Boicuk and A. Eremenko \cite[Theorem 3]{boer}.
\begin{theorem}$\cite{boer}$ \label{thmbe}
Let $f(z)=\sum_{k=0}^\infty c_kz^k$ be an entire function such that $|c_{k-1}c_{k+1}|\leqslant |c_k|^2$ and 
\begin{align*}
\limsup_{r\to\infty} \frac{\ln M(r;f)}{\ln^2r} = \beta < \infty.
\end{align*}
Then, $|c_k|\geqslant \exp\left\{-k^2/(4\beta)\right\}.$
\end{theorem}

\begin{proof}[Proof of Theorem \ref{thm2}]
To prove the statement, it suffices to show that, under the conditions \eqref{beta} and \eqref{lc}, the density $f$ satisfies the condition \eqref{C}.
Consider
\begin{align*}
\psi(z)=\sum_{j\in{\mathbb Z}} f(q^{-j})z^j = \sum_{j=-\infty}^{-1} f(q^{-j})z^j+\sum_{j=0}^\infty f(q^{-j})z^j=:\psi_1(z)+\psi_2(z).
\end{align*}
Here, $\psi_1$ is a function analytic at $\infty$ with $\psi_1(\infty)=0,$ whence $M(r;\psi_1)\to 0$ as $r\to\infty.$ Consequently,
\begin{align*}
\limsup_{r\to\infty} \frac{\ln M(r;\psi)}{\ln^2 r} = \limsup_{r\to\infty} \frac{\ln M(r;\psi_2)}{\ln^2r}.
\end{align*} 
For a $q$-density $f,$ one has $m_q(n;f)=(1-q)\psi(q^{-(n+1)})=(1-q)M(q^{-(n+1)};\psi)$ for all $n\in{\mathbb N}.$ Therefore,
\begin{multline*}
\limsup_{r\to\infty} \frac{\ln M(r;\psi_2)}{\ln^2 r} = \limsup_{r\to\infty} \frac{\ln M(r;\psi)}{\ln^2 r} = \limsup_{n\to\infty} \frac{\ln M(q^{-(n+1)};\psi)}{n^2\ln^2(1/q)}
\\=\limsup_{n\to\infty} \frac{\ln m_q(n;f)-\ln(1-q)}{n^2 \ln^2(1/q)} = \frac{A}{\ln^2(1/q)}< \infty.
\end{multline*} 
Thus, applying Theorem \ref{thmbe} to $\psi_2$ implies with $\beta=A/\ln^2(1/q)$ that 
\begin{align*}
f(q^{-j}) \geqslant \exp\left\{-\frac{j^2}{4\beta}\right\}=\exp\left\{-\frac{j^2\ln^2(1/q)}{4A}\right\} \geqslant
\exp\left\{-\frac{j^2\ln(1/q)}{2}\right\} = q^{j^2/2}
\end{align*}
due to \eqref{beta}. Since the condition \eqref{C} is satisfied, one derives the statement. 
\end{proof}
\begin{proof}[Proof of Theorem \ref{thm3}] Consider the entire function $\phi_m(z)=\prod_{j=1}^\infty (1-q^{mj}z)$ for which it is clear that $\phi_m(q^{-m(k+1)})=0$ for all $k\in{\mathbb N}_0.$
By Euler's identity \eqref{eu}
$$
\phi_m(z)=\sum_{j=0}^\infty \frac{(-1)^j q^{mj(j+1)/2}}{(q^m;q^m)_j}\, z^j
$$
which gives 
\begin{align}
\sum_{j=0}^\infty \frac{(-1)^j q^{mj(j+1)/2}}{(q^m;q^m)_j}\, q^{-m(n+1)j}=0, \quad n\in{\mathbb N}_0. \label{qmk}
\end{align}
Now, let $g$ be a $q$-density such that
\begin{align*}
g(q^{-mj})=f(q^{-mj})+\alpha \frac{(-1)^j q^{mj(j+1)/2}}{(q^m;q^m)_j}, \quad j\in{\mathbb N}_0
\end{align*}
and $g(q^{-j})=f(q^{-j})$ otherwise.
Note that, by condition \eqref{mj}, $\alpha>0$ can be chosen in such a way that $g(q^{-j})\geqslant 0$ for all $j\in{\mathbb Z}.$ Also, with the help of \eqref{qmk}, one derives
\begin{align*}
m_q(n;g)&=(1-q)\sum_{j\in{\mathbb Z}} g(q^{-j}) q^{-j(n+1)} \\
&=(1-q)\sum_{j=0}^{\infty} g(q^{-mj}) q^{-mj(n+1)}+(1-q)\sum_{j\in{\mathbb Z}, m\nmid j} g(q^{-j}) q^{-j(n+1)} \\
&=(1-q)\sum_{j=0}^{\infty} f(q^{-mj}) q^{-mj(n+1)}+\alpha(1-q)\sum_{j=0}^{\infty}\frac{(-1)^j q^{mj(j+1)/2}}{(q^m;q^m)_j}\, q^{-m(n+1)j}\\ & \qquad +(1-q)\sum_{m\nmid j} g(q^{-j}) q^{-j(n+1)} \\
&=(1-q)\sum_{j=0}^{\infty} f(q^{-mj}) q^{-mj(n+1)}+(1-q)\sum_{j\in{\mathbb Z}, m\nmid j} f(q^{-j}) q^{-j(n+1)} \\
&=(1-q)\sum_{j\in{\mathbb Z}} f(q^{-j}) q^{-j(n+1)}=m_q(n;f).
\end{align*}
Thus, $g \not\sim f$ while $m_q(n;g)=m_q(n;f)$ for all $n\in{\mathbb N}_0,$ which means that $P_X$ is $q$-moment indeterminate.
\end{proof}
Note that the result cannot be derived from Theorem \ref{thm2}, although 
$$\limsup_{n\to\infty} \frac{\ln m_q(n;f)}{n^2} \geqslant \frac{m}{2}\, \ln(1/q).$$

\begin{proof}[Proof of Proposition \ref{propcon}]  Let $\rho(t)$ be a probability density of $P_X.$ Given $q\in (0,1),$  one may write:
\begin{align*}
\mu_n = \int_{0}^\infty t^n \rho(t) dt = \sum_{j\in \mathbb{Z}} \int_{q^{j}}^{q^{j-1}} t^n\rho(t)dt 
\geqslant \sum_{j\in\mathbb{Z}} q^{jn} \left[ F_X(q^{j-1})-F_X(q^{j})\right].\\
\end{align*}
 To estimate the $q$-moments of $X$, recall that if $f$ the $q$-density  of $X$, then by definition \eqref{qmom}:
\begin{align*}m_q(n;X)=\int_0^\infty t^n f(t)d_qt.\end{align*} With the help of \eqref{qimpr}, one obtains:
\begin{align*}
m_q(n;X)=(1-q)\sum_{j\in\mathbb{Z}}q^{j(n+1)}f(q^j).
\end{align*}
 As $f$ is the $q$-derivative of $F_X,$  we obtain by virtue of \eqref{qder} that
\begin{align*}f(q^{j})=\frac{F_X(q^{j})-F_X(q^{j+1})}{q^{j}(1-q)}. \end{align*}
Hence,
$$ m_q(n;X)= \sum_{j\in\mathbb{Z}} q^{jn}\left[F_X(q^{j})-F_X(q^{j+1})\right]\leqslant 
q^{-n}\mu_n. $$
Therefore, by assumption \eqref{q1},
$$
\limsup_{n\to\infty} \frac{\ln m_q(n;X)}{n^2} \leqslant \limsup_{n\to\infty} \frac{n\ln(1/q)+\ln \mu_n }{n^2}\leqslant 
\frac{\ln(1/q_0)}{2}.
$$
By Proposition \ref{prop1}, $P_X$ is $q$-moment determinate whenever $q< q_0.$

\end{proof}
\begin{example} \label{hyper}
Let $f(t)$ be a density of a hyper-exponential distribution with parameters $\alpha,$ $\beta,$ $\gamma >0,$ that is,
$$
f(t)=\frac{\gamma \beta^{-\alpha/\gamma}}{\Gamma(\alpha/\gamma)}\, t^{\alpha-1} \exp\left(-x^{\gamma}/\beta\right), \quad t>0.
$$
It is known - see \cite[Section 11.4]{counter} - that the moments of this distribution are
$$
\mu_n=\frac{\beta^{n/\gamma}}{\Gamma(\alpha/\gamma)} \Gamma\left(\frac{n+\alpha}{\gamma}\right)
$$
and that for $\gamma\in(0, \frac12)$ the distribution is moment indeterminate and for $\gamma\in[\frac12, \infty)$ it is moment determinate. Since
$$
\limsup_{n\to\infty}\frac{\ln \mu_n}{n^2}=0,
$$
we conclude by Proposition \ref{propcon} that $X$ is $q$-moment determinate for all $q\in(0,1)$ regardless of parameter values.
\end{example}

\begin{proof}[Proof of Proposition \ref{propqexp}]
Let $F_X(t)=1-e_q(-\lambda t),$ $t\geqslant 0,$ whence the $q$-density of $F_X$ is $f(t)=\lambda e_q(-\lambda t).$ Meanwhile, the density of $F_X$ is 
\begin{align*}
\rho(t)&=\frac{d}{dt}\, F_X(t) = -e_q(-\lambda t) \frac{d}{dt} [\ln e_q(-\lambda t)]\\
&=-e_q(-\lambda t) \frac{d}{dt} \sum_{j=0}^{\infty} \ln \frac{1}{1+\lambda (1-q) q^j t}\\
&=\lambda (1-q) e_q(-\lambda t) \sum_{j=0}^{\infty} \frac{q^j}{1+\lambda (1-q) q^j t}=:\lambda (1-q) e_q(-\lambda t) S(t).
\end{align*}
To apply the Krein condition \cite[Section 11, p.101]{counter}, $\rho(t^2)$ will be estimated. By \eqref{zeng},
\begin{align*}
\frac{1}{e_q(-\lambda t^2)} = \prod_{j=0}^{\infty} \left(1+q^j\lambda(1-q)t^2\right) 
\leqslant C \exp\left\{\frac{2\ln^2 t}{\ln(1/q)}[1+\omega(t)]\right\},
\end{align*}
where $\omega(t)=o(1)$ as $t\to\infty.$ To estimate $S(t),$ consider $$\int_0^1 \frac{du}{1+au}=\frac{\ln(1+a)}{a}$$
and observe that, for $a>0,$
$$
\int_0^1 \frac{du}{1+au} \leqslant \sum_{j=0}^\infty \frac{q^j(1-q)}{1+aq^{j+1}} = \frac{1-q}{q} \left\{S(a)-\frac{1}{1+a}\right\},
$$
implying that 
$$
S(a) \geq \frac{1}{1+a}+\frac{q}{1-q} \frac{\ln(1+a)}{a} \geqslant C \frac{\ln(1+a)}{a}.
$$
Consequently,
\begin{align*}
\frac{1}{S(\lambda (1-q)t^2)} \leqslant C \frac{\lambda(1-q)t^2}{\ln(1+\lambda (1-q) t^2)}
\end{align*}
and
\begin{align*}
\frac{1}{\rho(t^2)} \leqslant C \frac{t^2}{\ln(1+\lambda (1-q) t^2)} \exp\left\{\frac{2\ln^2 t}{\ln(1/q)}[1+\omega(t)]\right\}
\end{align*}
yielding
\begin{align*}
-\ln \rho(t^2) \leqslant C \ln^2 t, \quad t \geqslant t_0.
\end{align*}
Therefore, the Krein integral
$$
\int_{t_0}^{\infty} \frac{-\ln \rho(t^2)}{1+t^2} \, dt < \infty
$$
and, by Krein's condition, the distribution $P_X$ is moment indeterminate for all $\lambda > 0$ and $0<q<1.$
\end{proof}

\section{Acknowledgements}The authors would like to extend their sincere gratitude to Prof. Alexandre Eremenko from Purdue University, USA for his valuable comments during the work on this paper.  Also, appreciations go to Mr. P. Danesh from the Atilim University Academic Writing and Advisory Centre for his help in the preparation of the manuscript.

\end{document}